\documentclass[12pt, a4paper]{amsart}

\usepackage[english]{babel}
\usepackage{amsmath}
\usepackage{amssymb}
\usepackage{bm}
\usepackage{amscd} %simple commutative diagrams
\usepackage{microtype} %avoids most bad boxes
\usepackage{verbatim} %for multiline comments
\usepackage{color}  %for comments
\usepackage{tikz-cd}\usetikzlibrary{babel} %for diagrams 
\usepackage{enumitem}
\usepackage{mathtools} %for coloneqq, ":="
\usepackage{cite}
\usepackage{hyperref} %pdflinks
\usepackage[initials,msc-links]{amsrefs} %backrefs?
\usepackage{dynkin-diagrams}

\usepackage[OT2, T1]{fontenc}
\title[A counterexample to determinant approximation]{A counterexample to the\\ determinant approximation conjecture}

\author[H. Kammeyer]{Holger Kammeyer}

 \address{Heinrich Heine University D{\"u}sseldorf, Faculty of Mathematics and Natural Sciences, Mathematical Institute, Germany}
 \email{holger.kammeyer@hhu.de}
 
\subjclass[2020]{46L10, 20C07}
\keywords{approximation, Fuglede--Kadison determinant}
\theoremstyle{plain}
\newtheorem{theorem}[equation]{Theorem}

\newtheorem{proposition}[equation]{Proposition}
\newtheorem{conjecture}[equation]{Conjecture}
\theoremstyle{definition}
\newtheorem{definition}[equation]{Definition}
\newtheorem*{definition*}{Definition}

\newtheorem*{observation*}{Observation}

\providecommand{\ignore}[1]{}

\providecommand{\R}{\mathbb{R}}
\providecommand{\Q}{\mathbb{Q}}
\providecommand{\Z}{\mathbb{Z}}

\providecommand{\C}{\mathbb{C}}

\newcommand*{\arXiv}[1]{ \href{http://www.arxiv.org/abs/#1}{arXiv:\textbf{#1}}}

\begin{document}

\begin{abstract}
  The determinant approximation conjecture states that for a residually finite group \(G\) and a chain \(G_i\) of finite index normal subgroups with trivial intersection, the Fuglede--Kadison determinant of every matrix over the group algebra \(\Q G\) is the limit of the Fuglede--Kadison determinants of the matrix reductions from \(\Q G\) to \(\Q (G/G_i)\).  We disprove the conjecture by exhibiting an explicit element of the integral group ring of the Heisenberg group with determinant equal to \(2\) and reduced determinants bounded by \(\sqrt[3]{5}\).
\end{abstract}

\maketitle

\section{Introduction}

The purpose of this article is to present a counterexample to the following standing conjecture from operator algebras and group theory.

\begin{conjecture} \label{conj:det-approximation}
  Let \(G\) be a discrete residually finite group and let
  \[ G = G_0 > G_1 > G_2 > \cdots \]
  be finite index normal subgroups of \(G\) such that \(\bigcap_{i=0}^\infty G_i = \{1\}\).  Let \(A \in M_{n \times m} (\Q G)\) be a matrix with canonical reductions \(A_i \in M_{n \times m}(\Q(G/G_i))\).   Then multiplication from the right defines bounded equivariant operators
  \[ R_A \colon (\ell^2 G)^n \rightarrow (\ell^2 G)^m, \qquad R_{A_i} \colon (\ell^2(G/G_i))^n \rightarrow (\ell^2(G/G_i))^m \]
whose Fuglede--Kadison determinants satisfy
  \[ \lim_{i \rightarrow \infty} \det\nolimits_{\mathcal{R}(G/G_i)} R_{A_i} = \det\nolimits_{\mathcal{R}(G)} R_A. \]
\end{conjecture}

The definition of Fuglede--Kadison determinants as well as the history and topological, arithmetic, and dynamical background of the conjecture will be presented below.

Recall that the \emph{discrete Heisenberg group} \(H = H_3(\Z)\) is the subgroup of \(\operatorname{SL}_3(\Z)\) consisting of all upper triangular \((3 \times 3)\)-matrices with ones on the diagonal and arbitrary integers as the remaining three entries.  The group \(H\) is 2-step nilpotent, torsion-free, and generated by
\[ a = \left(\begin{smallmatrix}
1 & 1 & 0\\
0 & 1 & 0\\
0 & 0 & 1
             \end{smallmatrix}\right),
           \qquad 
b = \left(\begin{smallmatrix}
1 & 0 & 0\\
0 & 1 & 1\\
0 & 0 & 1
\end{smallmatrix}\right). \]

\begin{theorem} \label{thm:counterexample}
  Let \(H\) be the discrete Heisenberg group and for \(i \ge 0\), let
  \[ H_i = \operatorname{ker} \left(H_3(\Z) \rightarrow H_3(\Z / 3^i \Z)\right) \]
  be the principal congruence subgroup of level \(3^i\).  Let \(A\) be the \((1 \times 1)\)-matrix whose only entry is
  \[ 1 -2a +2b \in \Z H. \]
  Then
    \[ \limsup_{i \rightarrow \infty} \det\nolimits_{\mathcal{R}(H/H_i)} R_{A_i} \le \sqrt[3]{5} < 2 = \det\nolimits_{\mathcal{R}(H)} R_A. \]
\end{theorem}

In particular, Conjecture~\ref{conj:det-approximation} is false even for \((1 \times 1)\)-matrices over integral group rings of nilpotent groups of nilpotency class two.

\subsection{Motivation for the conjecture}

Conjecture~\ref{conj:det-approximation} is of interest in equivariant topology, cohomology of arithmetic groups, and dynamical systems of algebraic origin.  We shall however take a biased viewpoint and present the topological motivation in some detail.  Afterwards, we will only briefly indicate the arithmetic and dynamical relevance albeit even more compelling.

The topological motivation is located in the theory of \emph{\(\ell^2\)-invariants}~\cites{Kammeyer:l2-invariants, Lueck:l2-invariants}.  Take a free \(G\)-CW complex \(X\) such that \(G \backslash X\) is compact and let \(C_*(X; \Q)\) be the rational cellular chain complex.  The \(G\)-action on \(X\) turns \(C_*(X; \Q)\) functorially into a chain complex of free \(\Q G\)-modules.  Choosing a suitable basis, the differentials are given by right multiplication with a matrix \(A \in M_{n \times m}(\Q G)\).  These matrices define bounded \(G\)-operators on the \emph{\(\ell^2\)-chain complex}
\[ C_*^{(2)}(X) = \ell^2 G \otimes_{\Q G} C_*(X; \Q). \]
Tools from operator algebras can be applied to extract useful invariants from this complex.  The best-studied example is the \(k\)-th \emph{\(\ell^2\)-Betti number} \(b_k^{(2)}(X)\): the \emph{von Neumann trace} of the kernel projection of the Laplacian
\[ {d^{(2)}_k}^* \!d^{(2)}_k + d^{(2)}_{k+1} {d^{(2)}_{k+1}\!\!}^*. \]
If \(G\) is trivial, \(b_k^{(2)}(X)\) reduces to the ordinary \(k\)-th Betti number \(b_k(X)\).

\smallskip
Another example is the \emph{\(\ell^2\)-torsion} \(\rho^{(2)}(X)\).  The classical counterpart is the \emph{Reidemeister torsion} \(\rho(X; V)\) associated with an orthogonal representation \(G \rightarrow O(V)\) on some finite-dimensional real inner product space \(V\) such that the chain complex
\[ C_*(X; V) = V \otimes_{\Q G} C_*(X; \Q) \]
is acyclic.  If \(d_n\) denotes the \(n\)-th differential in \(C_*(X; V)\), the Reidemeister torsion of \(X\) with respect to \(V\) can be defined by
\begin{equation} \label{eq:reidemeister-torsion} \rho(X; V) = \prod_{n \ge 0} \det\nolimits {|d_n|^\perp}^{(-1)^{n+1}}. \end{equation}
Here \(|d_n|^\perp\) is the restriction of the operator \(|d_n| = \sqrt{d_n^\ast d_n}\) to the orthogonal complement of its kernel, using that all \(C_*(X; V)\) have a well-defined inner product determined by \(V\).  Now the idea is to replace \(V\) with the right regular unitary representation \(\ell^2 G\).  However, the resulting \(\ell^2\)-chain complex \(C^{(2)}_*(X)\) is infinite dimensional whenever \(G\) is infinite so we need to specify what \(\det |d^{(2)}_n|^\perp\) should mean.  This is what the \emph{Fuglede--Kadison determinant} accomplishes.  We recall the definition in our setting.

\medskip
Let \(G\) be a countable group and for \(n, m \ge 1\), let \(T \colon (\ell^2 G)^n \rightarrow (\ell^2 G)^m\) be a left-\(G\)-equivariant bounded operator.  By continuous functional calculus, the positive operator \(|T| = \sqrt{T^*T}\) on \((\ell^2 G)^n\) defines a positive linear functional
\[ \Phi_{|T|} \colon C(\sigma(|T|), \C) \longrightarrow \C, \ f \mapsto \sum_{i=1}^n \langle e_i, f(|T|) e_i \rangle. \]
Here \(e_i \in (\ell^2 G)^n\) is the unit vector that has \(1 \in G\) as the i-th component and zero in all other components and \(C(\sigma(|T|), \C)\) denotes the continuous complex valued functions on the spectrum \(\sigma(|T|) \subset \R_{\ge 0}\) of \(|T|\).  By the Riesz Representation Theorem, the functional \(\Phi_{|T|}\) is given by integration with respect to a unique regular Borel measure \(\mu_{|T|}\) on \(\sigma(|T|)\) called the \emph{spectral measure} of \(|T|\).  We set \(\sigma(|T|)^+ = \sigma(|T|) \setminus \{0\}\).

\begin{definition} \label{def:determinant}
  The \emph{Fuglede--Kadison determinant} of \(T\) is defined by
  \[ \det\nolimits_{\mathcal{R}(G)} T = \exp\left(\int_{\sigma(|T|)^+} \!\!\log \,\textup{d}\mu_{|T|}\right). \]
\end{definition}

The notation \(\mathcal{R}(G)\) refers to the (right) group von Neumann algebra of \(G\) and is used here because the spectral measure can also be constructed as the von Neumann trace of the projection valued measure of \(|T|\).  Note that the Lebesgue integral in the definition is always defined because \(\sigma(|T|)^+ \subseteq (0, \lVert T \rVert]\).  A priori, the integral can take the value \(-\infty\) in which case we set \(\det_{\mathcal{R}(G)} T = \exp(-\infty) = 0\).

\medskip
Coming back to topology, we observe that the Fuglede--Kadison determinant \(\det_{\mathcal{R}(G)} d^{(2)}_n\) is the perfect infinite-dimensional replacement for \(\det |d_n|^\perp\).  So in view of Equation~\eqref{eq:reidemeister-torsion}, we define the \(\ell^2\)-torsion by
\[ \rho^{(2)}(X) = \rho^{(2)}(G \curvearrowright X) = \sum_{n \ge 0} (-1)^{n+1} \log \det\nolimits_{\mathcal{R}(G)} d^{(2)}_n. \]
The logarithm was applied so that some inherent properties of \(\ell^2\)-torsion can be described by additive instead of multiplicative formulas.  Note that \(\rho^{(2)}(X)\) depends on the specific CW structure of \(X\) unless we assume all \(\ell^2\)-Betti numbers of \(X\) vanish.  Without further mention, we assume that \(\det_{\mathcal{R}(G)} d_n^{(2)}\) is never zero for the given \(X\).

\smallskip
The question arises how an \(\ell^2\)-invariant relates to the original invariant.  Ideally, the \(\ell^2\)-invariant should be the limit of the classical invariants of the compact approximations \(G_i \backslash X\) for finite index normal subgroups \(G_i\) of \(G\).  In the case of \(\ell^2\)-Betti numbers, this works perfectly and is the content of the \emph{L\"uck approximation theorem}~\cite{Lueck:lueck-approximation}: For residually finite \(G\) and \((G_i)_{i \ge 0}\) a \emph{residual chain} as in Conjecture~\ref{conj:det-approximation}, we have
\[ \lim_{i \rightarrow \infty} \frac{b_n(G_i \backslash X)}{[G : G_i]} = b^{(2)}_n(X). \]

Inspired by this result, it was conjectured in~\cite{Lueck:homology-growth}*{Conjecture~1.11.(1)} that \(\ell^2\)-torsion should satisfy a corresponding approximation relation.
\begin{conjecture} \label{conj:l2-torsion-approx}
  Let \(G\) be residually finite, let \((G_i)_{i \ge 0}\) be a residual chain, and let \(X\) be a free \(G\)-CW complex such that \(G \backslash X\) is compact.  Then
  \[ \lim_{i \rightarrow \infty} \frac{\rho^{(2)}(\{1\} \curvearrowright G_i \backslash X)}{[G : G_i]} = \rho^{(2)}(X). \]
\end{conjecture}

But an immediate application of Theorem~\ref{thm:counterexample} gives the following result.

\begin{theorem} \label{thm:torsion-approx-failure}
  Let \(H\) and \(H_i\) be as in Theorem~\ref{thm:counterexample}.  There exists a connected free \(H\)-CW complex \(X\) with \(H \backslash X\) compact such that
  \[ \rho^{(2)}(X) - \limsup_{i \rightarrow \infty} \frac{\rho^{(2)}(\{1\} \curvearrowright H_i \backslash X)}{[H : H_i]} \ge \frac{1}{3} \log \frac{8}{5} > 0. \]
\end{theorem}

In particular, Conjecture~\ref{conj:l2-torsion-approx} is false.  The \(H\)-CW complex \(X\) appearing in the theorem is provided by \cite{Kammeyer:l2-invariants}*{Proposition~3.29} applied to our group ring element \(1 -2a +2b \in \Z H\).  Similar conjectures about approximation of \emph{analytic} and \emph{topological} \(\ell^2\)-torsion for a free cocompact Riemannian \(G\)-manifold \(X\)~\cite{Lueck:approximating-survey}*{Conjectures~8.2 and~8.3} state that
\[ \lim_{i \rightarrow \infty} \frac{\log \rho_{\textup{an}}(G_i \backslash X)}{[G : G_i]} = \rho^{(2)}_{\textup{an}}(X), \qquad \lim_{i \rightarrow \infty} \frac{\log \rho_{\textup{top}}(G_i \backslash X)}{[G : G_i]} = \rho^{(2)}_{\textup{top}}(X). \]
These conjectures should likewise fail by realizing \(1 -2a +2b \in \Z H\) as a differential of a suitable manifold, taking care that the contributions in the torsion invariants add rather than cancel by Poincar\'e duality.  We shall address the details in a forthcoming paper.

\smallskip
In contrast, we stress that our counterexample does \emph{not} refute \mbox{\(\ell^2\)-torsion} approximation of an \emph{aspherical} manifold with Heisenberg fundamental group.  In fact, for an aspherical \((2n+1)\)-dimensional closed connected manifold \(X\) with universal covering \(\widetilde{X}\), residually finite \(G = \pi_1 X\), and residual chain \((G_i)_{i \ge 0}\), we have the more specific conjecture that
\begin{equation} \label{eq:aspherical} \lim_{i \rightarrow \infty} \frac{\log |H_n(G_i \backslash \widetilde{X})_{\textup{tors}}|}{[G : G_i]} = (-1)^n \rho^{(2)}(\widetilde{X}), \end{equation}
as stated in~\cite{Lueck:homology-growth}*{Conjecture~1.12.(2)} and~\cite{Kammeyer:l2-invariants}*{Conjecture~6.22}.  W.\,L{\"u}ck showed that this conjecture is true if \(G\) is an elementary amenable group, like the Heisenberg group~\(H\), because then both sides of the equality are zero~\cite{Lueck:homology-growth}*{Corollary~1.13}.  Our counterexample does however disqualify the alleged proof strategy for this conjecture in~\cite{Kammeyer:l2-invariants}*{p.\,148} where it was suggested to attack this conjecture by proving the \emph{torsion Singer conjecture}, the \emph{small regulator conjecture}, and the \emph{determinant approximation conjecture}.

\medskip
On the arithmetic motivation for Conjecture~\ref{conj:det-approximation}, we only remark that the conjecture in~\eqref{eq:aspherical} on torsion growth in homology has an arithmetic cousin called the \emph{Bergeron--Venkatesh conjecture}~\cite{Bergeron-Venkatesh:asymptotic-growth}*{Conjecture~1.3} which applies to sequences of arithmetic congruence subgroups of anisotropic semisimple algebraic \(\Q\)-groups.  It seems inevitable that any potential proof of this conjecture will in one way or another have to address convergence of determinants.

\medskip
We finally turn to the dynamical motivation for Conjecture~\ref{conj:det-approximation}.  Any element \(f \in \Z G\) determines an \emph{algebraic dynamical system} given by the \(G\)-action on the compact Pontryagin dual \(X_f = \widehat{\Z G / \Z G f}\) of the discrete additive group \(\Z G / \Z G f\) endowed with Haar measure \(m_{X_f}\).  C.\,Deninger launched a program to express the \emph{entropy} \(h_f\) of such actions in terms of the Fuglede--Kadison determinant of \(f\) by the formula
\[ h_f = \log \det\nolimits_{\mathcal{R}(G)} f. \]
In~\cite{Deninger:entropy}, he showed this formula holds true for topological and measure-theoretical entropy if \(G\) is amenable, if \(R_f\) is positive, and if \(f\) is invertible in the Banach algebra \(\ell^1 G\) (ignoring another technical assumption).  Many authors have generalized this result to larger classes of \(G\) and/or weaker assumptions on \(f\)~\cites{Deninger-Schmidt:expansive, Bowen:entropy, Kerr-Li:entropy, Li:compact, Bowen-Li:harmonic, Li-Thom:entropy}.  All of the proofs have in common that they rely on some version of determinant approximation.  Invertibility assumptions on \(f\) have the virtue that finite-dimensional approximations have a uniform ``spectral gap'' that ensures the convergence of determinants.  We will come back to this point in the historical discussion below.

The definitive result completing Deninger's program was obtained by B.\,Hayes~\cite{Hayes:fuglede-kadison}*{Theorem~1.1.(ii)}.  He showed that for \emph{sofic} \(G\) with \emph{sofic approximation} \(\Sigma\), topological and measure-theoretical entropy satisfy
\[ h_{\Sigma, m_{X_f}} (X_f, G) = h_\Sigma(X_f, G) = \log \det\nolimits_{\mathcal{R}(G)} f \]
whenever \(R_f\) is injective.  Crucially, Hayes had the right intuition and predicted that under a mere injectivity assumption on~\(R_f\), determinant approximation should fail.  So he resorted to genuinely new methods to prove the theorem. 

\subsection{History of the conjecture}

Fuglede--Kadison determinants show up in the original proof of L{\"u}ck's approximation theorem~\cite{Lueck:lueck-approximation}, so the question of their convergence along finite quotient groups was likely already around at the time.  Therefore Conjecture~\ref{conj:det-approximation} is usually attributed to L{\"u}ck.  The first published reference I could locate is nonetheless a preprint of T.\,Schick~\cite{Schick:L2-determinant} from 1998 where the question is discussed in Remark~6.10 (the remark is omitted in the published article~\cite{Schick:L2-determinant-published}).  An analytic version of the question appeared earlier in 1995 in~\cite{Deitmar:geometric-zeta}*{Question~1.3.2, p.\,23}.  The definitive reference is in any case~\cite{Lueck:l2-invariants}*{Question~13.52, p.\,478}.  Ever since, the question has been discussed or mentioned by various authors: in chronological order \cite{Deninger-Schmidt:expansive}*{p.\,770}, \cite{Deninger:Mahler-measures}*{Question~21}, \cite{Bowen:entropy}, \cite{Kerr-Li:entropy}, \cite{Bowen-Li:harmonic}, \cite{Li-Thom:entropy}*{Section~3.1}, \cite{Koch-Lueck:graph}, \cite{Grabowski:large}*{Section~1.(i)}, \cite{Lueck:approximating-survey}*{Section~7}, \cite{Hayes:fuglede-kadison}*{Section~1}, \cite{Kammeyer:novikov-shubin}, \cite{Le:growth}, \cite{Lueck:twisting}*{Problem~10.5}, \cite{Kammeyer:l2-invariants}*{Remark~5.48 and Section~6.5}, \cite{Jaikin:l2-betti}*{Section~10.3}, \cite{Lueck:lehmer}*{Section~8.1}, \cite{Lueck:survey-l2-3}*{Section~3.5}, \cite{Kirstein-et-al:problems}*{Sections~7.2 and~7.3}.  At some point in this line, the ``question'' was promoted to a ``conjecture''.

\smallskip
Positive results on the conjecture are sparse so far.  Conjecture~\ref{conj:det-approximation} as stated is only known when \(G\) is (virtually) cyclic.  In that case, the conjecture transforms into an approximation property of the Mahler measure proven by Schmidt~\cite{Schmidt:dynamical-systems}*{Lemma~21.8, p.\,183}.  In contrast, the conjecture becomes wrong even for \(G = \Z\) once one replaces rational with complex coefficients in the group algebra.  This is illustrated by an example of L{\"u}ck~\cite{Lueck:l2-invariants}*{Example~13.69}.  Both observations borrow from (transcendental) number theory, a fact that was often taken as an indication that the conjecture should be difficult to prove if not wrong.

\smallskip
More satisfying partial results on Conjecture~\ref{conj:det-approximation} were obtained by imposing additional assumptions on the matrix \(A\) as we already indicated in the previous section.  Indeed, the definition of the Fuglede--Kadison determinant entails that the known weak convergence of spectral measures~\cite{Kammeyer:l2-invariants}*{Proposition~5.16} would grant the convergence of determinants if the measures had no support in a fixed punctured neighborhood of zero (``spectral gap'').  So it is the occurrence of too small singular values that might stand in the way of determinant approximation.  One way to dodge this issue is to assume that \(A \in \Q G\) is invertible in one way or another.  In this vein, C.\,Deninger and K.\,Schmidt~\cite{Deninger-Schmidt:expansive}*{Theorem~6.1} showed that determinant approximation works if \(A\) is a unit in the Banach algebra \(\ell^1(G)\).  This result was later strengthened by D.\,Kerr and H.\,Li~\cite{Kerr-Li:entropy}*{Theorem~7.3} to invertibility in the full group \(C^*\)-algebra \(C^*(G)\).  Requiring invertibility in the group von Neumann algebra, however, meaning \(R_A\) is an invertible operator, gives a spectral gap for \(R_A\) but as opposed to the \(C^*\)-case, this does not immediately produce a uniform bound for all the \(R_{A_i}\).  The remaining most difficult case arises when \(R_A\) is injective but non-invertible because then small singular values will occur.

\smallskip
An approach to control the small singular values was suggested by L{\"u}ck in~\cite{Lueck:approximating-survey}*{Section~17} where he introduces the \emph{uniform integrability condition} on the spectral distribution functions of the \(R_{A_i}\).  Under this condition, he proves the asserted convergence of determinants.  Theorem~\ref{thm:counterexample} thus shows that the uniform integrability condition can be violated for a matrix over the group ring.  In Section~\ref{section:uniform-integrability-failure}, we will reverse engineer and quantify this violation from the discrepancy of the determinants in our example.

\smallskip
It has thus been observed that other frameworks of determinant approximation appear to be more effective than approximation by finite quotient groups.  As an example, H.\,Li and A.\,Thom showed approximation of determinants of a positive matrix \(A\) with coefficients in \(\mathcal{R}(G)\) for an amenable group \(G\) along ``F{\o}lner windows''~\cite{Li-Thom:entropy}*{Theorem~1.4}.

\smallskip
Public opinions on the validity of Conjecture~\ref{conj:det-approximation} have varied in the past.  It seems there was consensus that the conjecture was unlikely to hold true in this generality.  However, some cautious optimism ``in many situations'' has been uttered for example in~\cite{Kirstein-et-al:problems}*{Remark~7.8}.  Also the paper \cite{Koch-Lueck:graph} was written with the notion to ``give evidence'' for Conjecture~\ref{conj:det-approximation}.  In contrast, B.\,Hayes, in the introduction of~\cite{Hayes:fuglede-kadison}, writes explicitly ``we suspect that this approximation is false in general''.  In view of the mentioned Bergeron--Venkatesh conjecture, it would be particularly interesting to know more about Conjecture~\ref{conj:det-approximation} for lattices in semisimple Lie groups.

\subsection{Background on the counterexample}  Let us now indicate how one could come up with the counterexample in Theorem~\ref{thm:counterexample}.  The historical discussion above gives some hints where one should start looking.  If \(G\) is virtually abelian, there seems to be some hope that Conjecture~\ref{conj:det-approximation} is true for (deep) number theoretical reasons.  So we should step away from abelian groups, but not too far away.  For otherwise, calculation of Fuglede--Kadison determinants could become unfeasible, both on the finite and infinite side.  So the Heisenberg group \(H\) comes to mind as it is not virtually abelian but still close to abelian, being 2-step nilpotent.  Additionally, it has an easy abstract presentation, which is useful for calculating Fuglede--Kadison determinants, as well as a nice matrix realization, which is useful for defining residual chains by congruence subgroups.

When looking for the right matrix \(A\), it is clear that one should start with a single group ring element \(A \in \Z H\) to facilitate computations.  Moreover, as we indicated above, we should make sure that small singular values occur for the \(R_{A_i}\), so the operator \(R_A\) should be injective but not invertible.  For computational purposes, it would be convenient if additionally all the \(R_{A_i}\) were injective (in that case \(R_A\) is injective anyway by L{\"u}ck approximation). We then have the simple formula
\[ \det\nolimits_{\mathcal{R}(G/G_i)} R_{A_i} = |\det\nolimits_\C R_{A_i}|^{\frac{1}{[G : G_i]}} \]
  for the finite approximations as we see from the discussion in~\cite{Kammeyer:l2-invariants}*{p.\,119}.

  It is well-known~\cite{Kammeyer:l2-invariants}*{Proposition~5.47} that the inequality
\[ \limsup_{i \rightarrow \infty} \det\nolimits_{\mathcal{R}(G/G_i)} R_{A_i} \le \det\nolimits_{\mathcal{R}(G)} R_A \]
holds true in general.  So it suits our purpose if we can keep the various \(\det\nolimits_{\mathcal{R}(G/G_i)} R_{A_i}\) small while \(\det\nolimits_{\mathcal{R}(G)} R_A\) is large: in any case \(> 1\).  Since \(\lVert R_{A_i} \rVert\) is uniformly bounded by some \(C > 1\), so are all singular values of \(R_{A_i}\).  It would thus be helpful if we could find large reducing subspaces \(W_i \subset \C (G/G_i)\) for \(R_{A_i}\) (meaning \(W_i\) is invariant under \(R_{A_i}\) and \(R_{A_i}^*\)) on which \(R_{A_i}\) restricts to a \(\C\)-linear map of determinant one.  For then
\[ \left| \det\nolimits_\C R_{A_i}\right| = \left|\det\nolimits_\C R_{A_i}\big|_{W_i^\perp} \right| \le C^{\dim_\C W_i^\perp}, \text{ so } \det\nolimits_{\mathcal{R}(G/G_i)} R_{A_i} \le C^{\frac{\dim_\C W_i^\perp}{[G : G_i]}}. \]
Hence if we manage to make \(\dim_\C W_i^\perp\) really small, there is hope to get this bound close to one.  At this point, a theorem of J.\,Boschheidgen from his recent PhD thesis comes in handy~\cite{Boschheidgen:limit}*{Theorem~1.2}.  We adapt it to our setting.

\begin{theorem}[Boschheidgen, 2024] \label{thm:boschheidgen}
  Let \(H\) be the discrete Heisenberg group, let \(p\) be an odd prime, and let \(H_i\) be the principal congruence subgroup of level \(p^i\).  For \(T = a - b \in \Z H\), we denote by \(d_i\) the dimension of the generalized eigenspace of the eigenvalue zero of the endomorphism
  \[ R_{T_i} \colon \C (H/H_i) \rightarrow \C(H/H_i).\]
  Then
  \[ \lim_{i \rightarrow \infty} \frac{d_i}{[H : H_i]} = \frac{p}{p+1}. \]
\end{theorem}

Boschheidgen therefore calls \(a-b \in \Z H\) an ``almost nilpotent'' element.  So we set
\[ A = 1 - 2(a-b) \in \Z H \]
and all our requests are met.  By Boschheidgen's theorem, \(R_{A_i} = I - 2R_{T_i}\) is unipotent on a large invariant portion of \(\C (H/H_i)\) where it thus restricts to an operator of determinant one.  The number ``2'' serves a double purpose.  Firstly, it makes sure that each \(R_{A_i}\) is invertible because the determinant is an odd integer.  Secondly, we will see that by a formula of Deninger, the ``2'' has the effect that \(\dim_{\mathcal{R}(H)} R_A = 2\).  Finally, the injective operator \(R_A\) is not invertible as one can see with a little effort from the methods in~\cite{Goell-et-al:wiener}.

The combination of these observations makes \(A\) an ideal candidate for a counterexample to Conjecture~\ref{conj:det-approximation}.  And indeed, we will show that this element \(A\) does the trick even for the residual chain \((H_i)\) with \(p=3\) and even without using the full generalized zero eigenspace of \(T_i\) for our estimates.

\subsection{Disclosure of AI assistance} Incidentally, I was a member of Jan Boschheidgen's PhD defense committee in Madrid in June 2023 upon invitation from his supervisor A.\,Jaikin-Zapirain.  When Jan was presenting his Theorem~\ref{thm:boschheidgen}, I remember thinking that this could be useful for the determinant approximation conjecture.  But during the conference on profinite rigidity the next week, the idea got lost in the shuffle.  Coming back to the conjecture recently, I found the counterexample with the help of OpenAI ChatGPT 6 Astra Pro after a long and detailed prompt including background material and potential strategies.  The model explicitly pointed out the relevance of Theorem~\ref{thm:boschheidgen} for the argument.  I second that and want to say that a large part of the credit is neither due to OpenAI nor to myself but to Jan Boschheidgen.  ChatGPT was additionally used to answer background questions but not for structuring or wording the article.  It was entirely written by myself.  I verified all arguments and remain responsible for the correctness of the text.  

\subsection{Outline of the article} The proof of Theorem~\ref{thm:counterexample} will be given in Sections~\ref{section:finite-determinants} and~\ref{section:infinite-determinant}.  Section~\ref{section:uniform-integrability-failure} debates the failure of uniform integrability in our example.  The proof of Theorem~\ref{thm:torsion-approx-failure} is presented in the final Section~\ref{section:torsion-approx-failure}.

\subsection{Acknowledgements}

I want to thank A.\,Jaikin-Zapirain and J.\,Boschheidgen for the invitation and the presentation of the thesis results in Madrid, W.\,L{\"u}ck for sparking my interest in Conjecture~\ref{conj:det-approximation} while I was a postdoc in Bonn, H.\,R{\"u}ping with whom I had run (inconclusive) computer experiments on \(\ell^2\)-torsion approximation at the time, and C.\,Deninger for answering questions on his calculation of Fuglede--Kadison determinants over the Heisenberg group.  This work was partially funded by the RTG ``Algebro-Geometric Methods in Algebra, Arithmetic, and Topology'' (DFG 284078965).

\section{The finite determinants} \label{section:finite-determinants}

In this section, we show that \(\det_{\mathcal{R}(H/H_i)} R_{A_i} \le \sqrt[3]{5}\) for all \(i \ge 1\).  Note that setting \(T = R_{a-b}\), the operator \(R_A\) is given by \(R_A = I - 2T\) where here and later \(I\) will always denote the respective identity operator.  Observe that the commutator \(c = a^{-1} b^{-1} a b\) is the matrix
\(c = \left(\begin{smallmatrix}
1 & 0 & 1\\
0 & 1 & 0\\
0 & 0 & 1
\end{smallmatrix}\right)\)
generating the center of \(H\).          

\medskip
Now fix \(i \ge 1\) and write for brevity \(q = 3^i\) and
\[ Q = H / H_i \cong H_3(\Z / q \Z), \]
so \(|Q| = q^3\).  For the moment, we will abuse notation and denote the images of \(a\), \(b\), \(c\) in \(Q\) by the same letters.  Correspondingly, \(T_q = R_{a-b}\) is an endomorphism of the finite-dimensional complex vector space \(\C Q\) and so is \(M_q = I - 2 T_q = R_{A_i}\).

\begin{proposition} \label{prop:det-m}
  The endomorphism \(M_q\) is invertible and we have
\[ \det\nolimits_{\mathcal{R}(Q)} M_q = \sqrt[q^3]{\left|\det\nolimits_\C M_q\right|}. \]
\end{proposition}

\begin{proof}
  The transformation matrix of \(T_q\) with respect to the basis \(Q\) has integer coefficients.  Hence so does the transformation matrix of \(M_q\) which is in fact congruent to the unit matrix mod~2.  Thus the usual determinant of \(M_q\) is an odd integer, whence \(M_q\) is invertible.  The second part of the proposition now follows from the discussion of Fuglede--Kadison determinants over finite groups in~\cite{Kammeyer:l2-invariants}*{p.\,119} because \(|\det_\C M_q|\) is the product of the singular values of \(M_q\) and all singular values are positive as \(M_q\) is invertible.
\end{proof}

The element \(c \in Q\) has order \(q\), so it acts by right multiplication on \(\C Q\) as an automorphism of order \(q\).  As such, it is diagonalizable and we obtain the eigenspace decomposition
\[ \C Q = \bigoplus_{\zeta^q = 1} V_\zeta, \qquad V_\zeta = \{ v \in \C Q \colon vc = \zeta v \}. \]

\begin{proposition} \label{prop:decomposition}
  The decomposition
  \[ \C Q = \bigoplus_{\zeta^q = 1} V_\zeta \]
  is orthogonal and consists of reducing subspaces for \(T_q\) and \(M_q\), each of dimension \(q^2\).
\end{proposition}

\begin{proof}
  The element \(c\) acts as a unitary hence normal transformation and the decomposition consists of eigenspaces with pairwise different eigenvalues, so the decomposition is orthogonal.  To see that each summand has dimension \(q^2\), we decompose \(Q\) as the disjoint union of the left cosets of \(Z = \langle c \rangle\) to see that the right \(Z\)-action on \(\C Q\) decomposes into \([Q : Z] = q^2\) copies of the right regular representation of \(Z\).  Each such copy is a direct sum of \(q\) one-dimensional eigenspaces of \(c\) with eigenvalues the \(q\) different \(q\)-th roots of unity.  Hence for each \(\zeta\) with \(\zeta^q = 1\), the subspace \(V_\zeta\) picks up a one-dimensional subspace from each copy of the regular \(Z\)-representation, whence \(\dim_\C V_\zeta = q^2\).

  Since \(c\) is central in \(Q\), the right multiplication operators of \(c\) and of any \(x \in \C Q\) commute, hence \(T_q = R_{a-b}\) and \(M_q = I - 2T_q\) preserve the eigenspace decomposition of \(c\) acting on \(\C Q\).  The adjoints are given by \(T_q^* = R_{a^{-1} -b^{-1}}\) and \(M_q^* = I - 2T_q^*\), so these operators preserve the decomposition as well.
\end{proof}

The proposition says in particular that \(T_q\) and \(M_q\) have block form with respect to the decomposition.  Next we verify that \(T_q\) is nilpotent of degree \(q\) on each block \(V_\zeta\) such that \(\zeta\) is primitive.  The argument is extracted from \cite{Boschheidgen:limit}*{Proposition~4.1}.

\begin{proposition} \label{prop:nilpotent}
  For every primitive \(q\)-th root of unity \(\zeta\), we have
  \[ T_q^q\big|_{V_\zeta} = 0. \]
\end{proposition}

\begin{proof}
  Rewrite the Heisenberg relation \(ab=bac\) as \(a^{-1} b a = bc^{-1}\).  Setting \(h = a^{-1}b\), we thus see that \(ha = ahc^{-1}\), so ``moving \(a\) in front of \(h\) comes at the expense of one \(c^{-1}\)''.  This shows that in the group algebra \(\C Q\), we have the identity
  \[ (a-b)^q = (a(1-h))^q = a^q \prod_{j=0}^{q-1} (1 -hc^{-j}) = \prod_{j=0}^{q-1} (1 -hc^{-j}) \]
  because clearly \(a^q = 1\) and because we could reorder the product as all factors commute by centrality of~\(c\).  We thus see that \((a-b)^q\) acts on \(V_\zeta\) as the operator
  \[ \prod_{j=0}^{q-1} \left(1 -\zeta^{-j} R_h\right). \]
  Since \(\zeta\) is primitive and \(q\) is odd, we have the polynomial identity
  \[ \prod_{j=0}^{q-1}(1 - \zeta^{-j} X) = 1 - X^q. \]
  Thus \((a-b)^q\) acts in fact as \(I - R_h^q\) on \(V_\zeta\).  But the quick integer matrix calculation
\[
\left(\left(\begin{smallmatrix}
1 & 1 & 0\\
0 & 1 & 0\\
0 & 0 & 1
\end{smallmatrix}\right)^{-1} \!\!\cdot
\left(\begin{smallmatrix}
1 & 0 & 0\\
0 & 1 & 1\\
0 & 0 & 1
\end{smallmatrix}\right)\right)^q =
\left(\begin{smallmatrix}
1 & -q & -\frac{q(q+1)}{2} \\
0 & 1 & q \\
0 & 0 & 1
\end{smallmatrix}\right)
\]
shows that also \(h^q = 1\) because \(q\) is odd.  Thus \((a-b)^q\) acts as the zero operator and the proof is complete.
\end{proof}

We have now collected all necessary preliminaries to estimate the determinants.  First note that there are \(\varphi(q) = \frac{2q}{3}\) primitive \(q\)-th roots of unity.  So the reducing subspace
\[ W_q = \!\bigoplus_{\zeta^q = 1 \text{ primitive}}\! V_\zeta \]
of \(M_q\) and \(T_q\) has dimension \(\frac{2q^3}{3}\) by Proposition~\ref{prop:decomposition}.  Since \(T_q\) is nilpotent on \(W_q\) by Proposition~\ref{prop:nilpotent}, we conclude that \(M_q\big|_{W_q}\) is unipotent, hence
\begin{equation} \label{eq:det-w} \det\nolimits_\C \,M_q\big|_{W_q} = 1. \end{equation}

On the orthogonal complement \(W_q^\perp\), we use a crude estimate for the determinant.  We observe that \(M_q\) is by definition a sum of five unitaries, hence the operator norm satisfies \(\lVert M_q \rVert \le 5\) by the triangle inequality.  Thus every singular value of \(M_q\big|_{W_q^\perp}\) is at most five and since \(\dim_\C W_q^\perp = \frac{q^3}{3}\), we obtain
\begin{equation} \label{eq:det-w-perp} \left| \det\nolimits_\C \,M_q\big|_{W_q^\perp} \right| \le 5^{\frac{q^3}{3}}. \end{equation}
Propositions~\ref{prop:det-m} and~\ref{prop:decomposition} and Equations~\eqref{eq:det-w} and~\eqref{eq:det-w-perp} thus show that
\[ \det\nolimits_{\mathcal{R}(H/H_i)} R_{A_i} = \det\nolimits_{\mathcal{R}(Q)} M_q \le \sqrt[3]{5}. \]
  Since \(i \ge 1\) was arbitrary, this also shows that
\[ \limsup_{i \rightarrow \infty} \det\nolimits_{\mathcal{R}(H/H_i)} R_{A_i}  \le \sqrt[3]{5}. \]
  
\section{The infinite determinant} \label{section:infinite-determinant}

We now show \(\det_{\mathcal{R}(H)} R_A = 2\).  To do so, we return to the original meaning of \(a\), \(b\), \(c\) as elements in the discrete Heisenberg group~\(H\).  To calculate the Fuglede--Kadison determinant of the operator \(R_A\), we rely on a result of C.\,Deninger~\cite{Deninger:determinants}*{Theorem~11\,(a)}.  It states that for the abstractly presented Heisenberg group
\[ H = \left\langle x, y, z \,\big|\, y^{-1}x^{-1}yxz^{-1}, \ x^{-1}z^{-1}xz, \ y^{-1}z^{-1}yz \right\rangle \]
and for any complex polynomial \(P(z_1, z_2)\) in two commuting variables, the element \(B = 1 -P(y,z)x \in \C H\) satisfies
\[ \log \det\nolimits_{\mathcal{R}(H)} R_B = \int_{S^1} \left(\int_{S^1} \log |P(z_1,z_2)| \,\textup{d}\mu(z_1)\right)^+ \textup{d}\mu(z_2). \]
Here \((\,\cdot\,)^+ = \max\{0, \,\cdot\,\}\) and \(\mu\) denotes the normalized Haar measure on the circle~\(S^1 \subset \C\).  We remark that in this formula, as opposed to our Definition~\ref{def:determinant}, the appearing Fuglede--Kadison determinant is understood in the original sense so that zero is included in the domain of integration.  But since we already know from Proposition~\ref{prop:det-m} and L{\"u}ck approximation that \(R_A\) is injective, the two definitions coincide for \(R_A\).

One readily checks that setting \(x = a\) and \(y = ba^{-1}\), hence \(z = ab^{-1}a^{-1}b = c^{-1}\) realizes the abstract presentation of \(H\) from above as our matrix Heisenberg group \(H_3(\Z)\).  Indeed, the relations are satisfied and conversely, mapping \(a\) to \(x\) and \(b\) to \(yx\) defines an inverse because the two relations saying that \(ab^{-1}a^{-1}b\) is central in \(H_3(\Z)\) are granted under this assignment because \(ab^{-1}a^{-1}b\) maps to \(z\).

Under this identification, our group ring element \(1 -2a +2b\) is the element \(B \in \C H\) arising from the polynomial \(P(z_1, z_2) = 2(1 - z_1)\).  As \(P\) is constant in \(z_2\), the outer integral can be discarded by the normalization of \(\mu\).  The inner integral becomes
\[ \log 2 + \int_{S^1} \log |1-z_1| \,\textup{d}\mu(z_1) \]
and the second summand vanishes as it is just the logarithmic Mahler measure of the polynomial \((z-1)\).  This shows
\[ \det\nolimits_{\mathcal{R}(H)} R_A = 2. \]

\section{Failure of uniform integrability} \label{section:uniform-integrability-failure}
  
In this section, we apply the estimates \(\det\nolimits_{\mathcal{R}(H/H_i)} R_{A_i}  \le \sqrt[3]{5}\) and \(\det\nolimits_{\mathcal{R}(H)} R_A = 2\) to make the violation of L{\"u}ck's \emph{uniform integrability condition}~\cite{Lueck:approximating-survey}*{Theorem~17.1.(v)} explicit in our example.

\smallskip
Let \(A \in M_{n \times m}(\Q G)\) be a matrix such that all \(R_{A_i}\) are injective for some fixed residual chain \((G_i)_{i \ge 0}\).  Let \(\mu = \mu_{|R_A|}\) and \(\mu_i = \mu_{|R_{A_i}|}\) be the associated spectral measures.  Uniform integrability would say that there exists \(\varepsilon > 0\) such that the functions
  \[ f_i \colon (0, \varepsilon] \rightarrow \R_{\ge 0}, \quad f_i(\lambda) = \frac{\mu_i((0,\lambda])}{\lambda} \]
  have a uniform integrable upper bound \(g \colon (0, \varepsilon] \rightarrow \R_{\ge 0}\) for the standard Lebesgue measure on \((0, \varepsilon]\).  However, the Fubini--Tonelli theorem gives
  \begin{align*} \int_{0^+}^\varepsilon f_i(\lambda) \,\textup{d}\lambda &= \int_{0^+}^\varepsilon \left(\int_{(0,\varepsilon]} \frac{\mathbf{1}_{(0, \lambda]}(\nu)}{\lambda}\, \textup{d}\mu_i(\nu) \right) \textup{d}\lambda \\
                                                                       &=  \int_{(0,\varepsilon]} \left(\int_{0^+}^\varepsilon \frac{\mathbf{1}_{(0, \lambda]}(\nu)}{\lambda}\, \textup{d}\lambda  \right) \textup{d}\mu_i(\nu) \\
                                                                       &=  \int_{(0,\varepsilon]} \left(\int_\nu^\varepsilon \frac{1}{\lambda}\, \textup{d}\lambda  \right) \textup{d}\mu_i(\nu) = \int_{(0,\varepsilon]}  \log \frac{\varepsilon}{\nu} \ \textup{d}\mu_i(\nu).
  \end{align*}
  For the latter integral, we have the asymptotic estimate
  \begin{align*} \liminf_{i \rightarrow \infty} &\int_{(0,\varepsilon]}  \log \frac{\varepsilon}{\nu} \ \textup{d}\mu_i(\nu) = \\
                                                & \liminf_{i \rightarrow \infty} \left( \int_{(0,\infty)} \log (\max \{\varepsilon, \nu\}) \,\textup{d}\mu_i(\nu) - \int_{(0,\infty)} \log (\nu) \,\textup{d}\mu_i(\nu)\right).
  \end{align*}
  Now note that the function \(\log (\max\{\varepsilon, \nu\})\) in \(\nu\) (as opposed to \(\log \nu\)) is continuous on a fixed compact interval \([0,C]\) containing the supports of \(\mu\) and all the \(\mu_i\) and we have \(\mu_i(\{0\}) = 0\) hence \(\mu(\{0\}) = 0\) by L{\"u}ck approximation.  Thus weak convergence implies
  \begin{align*} \liminf_{i \rightarrow \infty} &\int_{(0,\varepsilon]}  \log \frac{\varepsilon}{\nu} \ \textup{d}\mu_i(\nu) \ge \\
                                                &\int_{(0, \infty)} \log(\max\{\varepsilon, \nu\}) \ \textup{d}\mu(\nu) - \limsup_{i \rightarrow \infty} \log \det\nolimits_{\mathcal{R}(G/G_i)} R_{A_i}.
  \end{align*}
  By monotonicity of the Lebesgue integral, we can replace the last integral in this inequality with \(\log \det_{\mathcal{R}(G)} R_A\). We conclude
  \[ \liminf_{i \rightarrow \infty} \int_{0^+}^\varepsilon f_i(\lambda) \,\textup{d}\lambda \ge \log \det\nolimits_{\mathcal{R}(G)} R_A - \limsup_{i \rightarrow \infty} \log \det\nolimits_{\mathcal{R}(G/G_i)} R_{A_i} \ge 0. \]
  Assuming that the last inequality is strict, we thus obtain a positive lower bound independent of~\(\varepsilon\).  Clearly, this shows that no integrable dominant function \(g\) exists for the family \((f_i)\).

  \smallskip
  In our example \(A = 1 -2a +2b \in \Z H\), we have seen in Proposition~\ref{prop:det-m} that each \(R_{A_i}\) is injective for the residual chain \((H_i)_{i \ge 0}\).  So the above arguments apply and we obtain
  \[ \liminf_{i \rightarrow \infty} \int_{0^+}^\varepsilon f_i(\lambda) \,\textup{d}\lambda \ge \log 2 - \frac{1}{3} \log 5 \approx 0.156\ldots\,.\]

\section{Failure of torsion approximation} \label{section:torsion-approx-failure}

In this section, we prove Theorem~\ref{thm:torsion-approx-failure}.  So let \(H\) and \(H_i\) be as in Theorem~\ref{thm:counterexample}.  Let \(X\) be the free \(H\)-CW complex that \cite{Kammeyer:l2-invariants}*{Proposition~3.29} associates with \(A = 1 -2a +2b \in \Z H\) and with the two generators \(a, b\) of \(H\).  By construction, \(X\) is connected, \(H \backslash X\) is compact, and the cellular chain complex takes the form
\[ 0 \longrightarrow \Z H \xrightarrow{\ \cdot A \ } \Z H \xrightarrow{\ 0 \ } (\Z H)^2 \xrightarrow{\ \cdot B \ } \Z H \longrightarrow 0. \]
Note that by our definition, we have \(\det_{\mathcal{R}(H)}(0) = 1\), so
\[ \rho^{(2)}(X) = \log \det\nolimits_{\mathcal{R}(H)} R_B +  \log \det\nolimits_{\mathcal{R}(H)} R_A. \]
Since as in~\cite{Kammeyer:l2-invariants}*{p.\,149}, we have
\[ \frac{\rho^{(2)}(\{1\} \curvearrowright H_i \backslash X)}{[H : H_i]} = \rho^{(2)}(H/H_i \curvearrowright H_i \backslash X), \]
we obtain similarly
\[ \frac{\rho^{(2)}(\{1\} \curvearrowright H_i \backslash X)}{[H : H_i]} = \log \det\nolimits_{\mathcal{R}(H/H_i)} R_{B_i} + \log \det\nolimits_{\mathcal{R}(H/H_i)} R_{A_i}. \]
From Theorem~\ref{thm:counterexample}, we have
\[ \det\nolimits_{\mathcal{R}(H)} R_A = 2, \qquad \det\nolimits_{\mathcal{R}(H/H_i)} R_{A_i} \le \sqrt[3]{5}. \]
Additionally, we obtain from~\cite{Kammeyer:l2-invariants}*{Proposition~5.47} and from continuity and monotonicity of the logarithm that
\[ \limsup_{i \rightarrow \infty} \log \det\nolimits_{\mathcal{R}(H/H_i)} R_{B_i} \le \log \det\nolimits_{\mathcal{R}(H)} R_B. \]
Putting pieces together, we conclude
\begin{align*}
  \rho^{(2)}(X)& - \limsup_{i \rightarrow \infty} \frac{\rho^{(2)}(\{1\} \curvearrowright H_i \backslash X)}{[H : H_i]} \ge \\
                  &\log \det\nolimits_{\mathcal{R}(H)} R_A - \limsup_{i \rightarrow \infty} \log \det\nolimits_{\mathcal{R}(H/H_i)} R_{A_i} \ge \\
                  &\log 2 - \log \sqrt[3]{5} = \frac{1}{3} \log \frac{8}{5}.
  \end{align*}


\begin{bibdiv}[References]

  \begin{biblist}
    \bib{Bergeron-Venkatesh:asymptotic-growth}{article}{
   author={Bergeron, Nicolas},
   author={Venkatesh, Akshay},
   title={The asymptotic growth of torsion homology for arithmetic groups},
   journal={J. Inst. Math. Jussieu},
   volume={12},
   date={2013},
   number={2},
   pages={391--447},
   issn={1474-7480},
   review={\MR{3028790}},
%   doi={10.1017/S1474748012000667},
 }

 \bib{Boschheidgen:limit}{article}{
   author={Boschheidgen, Jan},
   title={On limit eigenvalue distributions associated to residual chains of
   groups},
   journal={J. Operator Theory},
   volume={92},
   date={2024},
   number={2},
   pages={439--452},
   issn={0379-4024},
   review={\MR{4849102}},
%   doi={10.7900/jot.2022nov02.2406},
 }
 
 \bib{Bowen:entropy}{article}{
  author={Bowen, Lewis},
  title={Entropy for expansive algebraic actions of residually finite groups},
  journal={Ergodic Theory Dynam. Systems},
  volume={31},
  date={2011},
  number={3},
  pages={703--718},
 review={\MR{2794944}},
}

\bib{Bowen-Li:harmonic}{article}{
   author={Bowen, Lewis},
   author={Li, Hanfeng},
   title={Harmonic models and spanning forests of residually finite groups},
   journal={J. Funct. Anal.},
   volume={263},
   date={2012},
   number={7},
   pages={1769--1808},
   issn={0022-1236},
   review={\MR{2956925}},
%   doi={10.1016/j.jfa.2012.06.015},
 }

 \bib{Deitmar:geometric-zeta}{article}{
    author={Deitmar, Anton},
    title={Geometric Zeta Functions, \(L^2\)-Theory, and Compact Shimura Manifolds},
   date={1995},
   review={\arXiv{math/9503216}},
 }

\bib{Deninger:determinants}{article}{
   author={Deninger, Christopher},
   title={Determinants on von Neumann algebras, Mahler measures and Ljapunov
   exponents},
   journal={J. Reine Angew. Math.},
   volume={651},
   date={2011},
   pages={165--185},
   issn={0075-4102},
   review={\MR{2774314}},
%   doi={10.1515/CRELLE.2011.012},
}

\bib{Deninger:entropy}{article}{
   author={Deninger, Christopher},
   title={Fuglede-Kadison determinants and entropy for actions of discrete
   amenable groups},
   journal={J. Amer. Math. Soc.},
   volume={19},
   date={2006},
   number={3},
   pages={737--758},
   issn={0894-0347},
   review={\MR{2220105}},
%   doi={10.1090/S0894-0347-06-00519-4},
 }

\bib{Deninger:Mahler-measures}{article}{
  author={Deninger, Christopher},
  title={Mahler measures and Fuglede--Kadison determinants},
  journal={M{\"u}nster J. Math.},
  volume={2},
  date={2009},
  pages={45--63},
  review={\MR{2545607}},
}

\bib{Deninger-Schmidt:expansive}{article}{
   author={Deninger, Christopher},
   author={Schmidt, Klaus},
   title={Expansive algebraic actions of discrete residually finite amenable
   groups and their entropy},
   journal={Ergodic Theory Dynam. Systems},
   volume={27},
   date={2007},
   number={3},
   pages={769--786},
   issn={0143-3857},
   review={\MR{2322178}},
%  doi={10.1017/S0143385706000939},
 }

\bib{Goell-et-al:wiener}{article}{
   author={G\"oll, Martin},
   author={Schmidt, Klaus},
   author={Verbitskiy, Evgeny},
   title={A Wiener lemma for the discrete Heisenberg group},
   journal={Monatsh. Math.},
   volume={180},
   date={2016},
   number={3},
   pages={485--525},
   issn={0026-9255},
   review={\MR{3513217}},
%   doi={10.1007/s00605-016-0894-0},
}

 \bib{Grabowski:large}{article}{
   author={Grabowski, \L ukasz},
   title={Group ring elements with large spectral density},
   journal={Math. Ann.},
   volume={363},
   date={2015},
   number={1-2},
   pages={637--656},
   issn={0025-5831},
   review={\MR{3394391}},
%   doi={10.1007/s00208-015-1170-7},
 }

 \bib{Hayes:fuglede-kadison}{article}{
   author={Hayes, Ben},
   title={Fuglede-Kadison determinants and sofic entropy},
   journal={Geom. Funct. Anal.},
   volume={26},
   date={2016},
   number={2},
   pages={520--606},
   issn={1016-443X},
   review={\MR{3513879}},
%   doi={10.1007/s00039-016-0370-y},
 }
 
 \bib{Jaikin:l2-betti}{article}{
   author={Jaikin-Zapirain, Andrei},
   title={$L^2$-Betti numbers and their analogues in positive
   characteristic},
   conference={
      title={Groups St Andrews 2017 in Birmingham},
   },
   book={
      series={London Math. Soc. Lecture Note Ser.},
      volume={455},
%      publisher={Cambridge Univ. Press, Cambridge},
   },
   isbn={978-1-108-72874-4},
   date={2019},
   pages={346--405},
   review={\MR{3931420}},
 }

\bib{Kammeyer:l2-invariants}{book}{
   author={Kammeyer, Holger},
   title={Introduction to $\ell^2$-invariants},
   series={Lecture Notes in Mathematics},
   volume={2247},
   publisher={Springer, Cham},
   date={2019},
   pages={viii+181},
   isbn={978-3-030-28296-7},
   isbn={978-3-030-28297-4},
   review={\MR{3971279}},
%   doi={10.1007/978-3-030-28297-4},
 }
 
\bib{Kammeyer:novikov-shubin}{article}{
   author={Kammeyer, Holger},
   title={Approximating Novikov-Shubin numbers of virtually cyclic
   coverings},
   journal={Groups Geom. Dyn.},
   volume={11},
   date={2017},
   number={4},
   pages={1231--1251},
   issn={1661-7207},
   review={\MR{3737281}},
%   doi={10.4171/GGD/427},
 }
 
 \bib{Kerr-Li:entropy}{article}{
   author={Kerr, David},
   author={Li, Hanfeng},
   title={Entropy and the variational principle for actions of sofic groups},
   journal={Invent. Math.},
   volume={186},
   date={2011},
   number={3},
   pages={501--558},
   issn={0020-9910},
   review={\MR{2854085}},
%   doi={10.1007/s00222-011-0324-9},
 }

 \bib{Kirstein-et-al:problems}{article}{
   author={Kirstein, Dominik},
   author={Kremer, Christian},
   author={L\"uck, Wolfgang},
   title={Some problems and conjectures about $L^2$-invariants},
   conference={
      title={Geometry and topology of aspherical manifolds},
   },
   book={
      series={Contemp. Math.},
      volume={816},
      publisher={Amer. Math. Soc., Providence, RI},
   },
   isbn={978-1-4704-7495-9},
   isbn={[9781470478698]},
   date={2025},
   pages={3--43},
   review={\MR{4885752}},
%   doi={10.1090/conm/816/16349},
}

 \bib{Koch-Lueck:graph}{article}{
   author={Koch, Herbert},
   author={L\"uck, Wolfgang},
   title={On the spectral density function of the Laplacian of a graph},
   journal={Expo. Math.},
   volume={32},
   date={2014},
   number={2},
   pages={178--189},
   issn={0723-0869},
   review={\MR{3206650}},
%   doi={10.1016/j.exmath.2013.09.001},
 }

 \bib{Le:growth}{article}{
   author={L\^e, Thang T. Q.},
   title={Growth of homology torsion in finite coverings and hyperbolic
   volume},
%   language={English, with English and French summaries},
   journal={Ann. Inst. Fourier (Grenoble)},
   volume={68},
   date={2018},
   number={2},
   pages={611--645},
   issn={0373-0956},
   review={\MR{3803114}},
%   doi={10.5802/aif.3173},
 }

 \bib{Li:compact}{article}{
   author={Li, Hanfeng},
   title={Compact group automorphisms, addition formulas and Fuglede-Kadison
   determinants},
   journal={Ann. of Math. (2)},
   volume={176},
   date={2012},
   number={1},
   pages={303--347},
   issn={0003-486X},
   review={\MR{2925385}},
%   doi={10.4007/annals.2012.176.1.5},
}
 
 \bib{Li-Thom:entropy}{article}{
   author={Li, Hanfeng},
   author={Thom, Andreas},
   title={Entropy, determinants, and $L^2$-torsion},
   journal={J. Amer. Math. Soc.},
   volume={27},
   date={2014},
   number={1},
   pages={239--292},
   issn={0894-0347},
   review={\MR{3110799}},
%   doi={10.1090/S0894-0347-2013-00778-X},
 }
 
 \bib{Lueck:homology-growth}{article}{
   author={L\"uck, Wolfgang},
   title={Approximating $L^2$-invariants and homology growth},
   journal={Geom. Funct. Anal.},
   volume={23},
   date={2013},
   number={2},
   pages={622--663},
   issn={1016-443X},
   review={\MR{3053758}},
%  doi={10.1007/s00039-013-0218-7},
 }
 
\bib{Lueck:approximating-survey}{article}{
   author={L\"uck, Wolfgang},
   title={Approximating $L^2$-invariants by their classical counterparts},
   journal={EMS Surv. Math. Sci.},
   volume={3},
   date={2016},
   number={2},
   pages={269--344},
   issn={2308-2151},
   review={\MR{3576534}},
%   doi={10.4171/EMSS/18},
 }

 \bib{Lueck:lueck-approximation}{article}{
   author={L\"uck, Wolfgang},
   title={Approximating $L^2$-invariants by their finite-dimensional
   analogues},
   journal={Geom. Funct. Anal.},
   volume={4},
   date={1994},
   number={4},
   pages={455--481},
   issn={1016-443X},
   review={\MR{1280122}},
%  doi={10.1007/BF01896404},
 }
 
\bib{Lueck:l2-invariants}{book}{
   author={L\"uck, Wolfgang},
   title={$L^2$-invariants: theory and applications to geometry and
   $K$-theory},
   series={Ergebnisse der Mathematik und ihrer Grenzgebiete. 3. Folge. A
   Series of Modern Surveys in Mathematics},
   volume={44},
   publisher={Springer-Verlag, Berlin},
   date={2002},
   pages={xvi+595},
   isbn={3-540-43566-2},
   review={\MR{1926649}},
%   doi={10.1007/978-3-662-04687-6},
}

\bib{Lueck:lehmer}{article}{
   author={L\"uck, Wolfgang},
   title={Lehmer's problem for arbitrary groups},
   journal={J. Topol. Anal.},
   volume={14},
   date={2022},
   number={4},
   pages={901--932},
   issn={1793-5253},
   review={\MR{4523563}},
%   doi={10.1142/S1793525321500035},
 }

 \bib{Lueck:survey-l2-3}{article}{
   author={L\"uck, Wolfgang},
   title={Survey on $L^2$-invariants and 3-manifolds},
   journal={Bull. Lond. Math. Soc.},
   volume={53},
   date={2021},
   number={6},
   pages={1583--1620},
   issn={0024-6093},
   review={\MR{4368687}},
%   doi={10.1112/blms.12536},
 }
 
\bib{Lueck:twisting}{article}{
   author={L\"uck, Wolfgang},
   title={Twisting $L^2$-invariants with finite-dimensional representations},
   journal={J. Topol. Anal.},
   volume={10},
   date={2018},
   number={4},
   pages={723--816},
   issn={1793-5253},
   review={\MR{3881040}},
%   doi={10.1142/S1793525318500279},
 }
 
\bib{Schick:L2-determinant}{article}{
    author={Schick, Thomas},
    title={$L^2$-determinant class and approximation of $L^2$-Betti numbers},
   date={1998},
   review={\arXiv{math/9807032v1}},
 }

 \bib{Schick:L2-determinant-published}{article}{
   author={Schick, Thomas},
   title={$L^2$-determinant class and approximation of $L^2$-Betti numbers},
   journal={Trans. Amer. Math. Soc.},
   volume={353},
   date={2001},
   number={8},
   pages={3247--3265},
   issn={0002-9947},
   review={\MR{1828605}},
%   doi={10.1090/S0002-9947-01-02699-X},
}

  \bib{Schmidt:dynamical-systems}{book}{
   author={Schmidt, Klaus},
   title={Dynamical systems of algebraic origin},
   series={Modern Birkh\"auser Classics},
   note={2011 reprint of the 1995 original},
   publisher={Birkh\"auser/Springer Basel AG, Basel},
   date={1995},
   pages={xviii+310},
   isbn={978-3-0348-0276-5},
   isbn={978-3-0348-0277-2},
   review={\MR{3024809}},
 }

  \end{biblist}
\end{bibdiv}
\end{document}